\documentclass[10pt]{amsart}
\usepackage[initials]{amsrefs}
\usepackage{amscd,amsmath,amsfonts,amssymb}

\providecommand{\U}[1]{\protect\rule{.1in}{.1in}}

\newtheorem{theorem}{Theorem}[section]
\newtheorem{proposition}[theorem]{Proposition}
\newtheorem{corollary}[theorem]{Corollary}

\numberwithin{equation}{section}

\begin{document}
\title{H\"{o}lder's inequality: some recent and unexpected applications}
\author[Albuquerque]{N. Albuquerque}
\address{Departamento de Matem\'{a}tica, \\
\indent Universidade Federal da Para\'{\i}ba, \\
\indent 58.051-900 - Jo\~{a}o Pessoa, Brazil.}
\email{ngalbqrq@gmail.com}
\author[Ara\'{u}jo]{G. Ara\'{u}jo}
\address{Departamento de Matem\'{a}tica, \\
\indent Universidade Federal da Para\'{\i}ba, \\
\indent
58.051-900 - Jo\~{a}o Pessoa, Brazil.}
\email{gdasaraujo@gmail.com}
\author[Pellegrino]{D. Pellegrino}
\address{Departamento de Matem\'{a}tica, \\
\indent Universidade Federal da Para\'{\i}ba, \\
\indent
58.051-900 - Jo\~{a}o Pessoa, Brazil.}
\email{dmpellegrino@gmail.com and pellegrino@pq.cnpq.br}
\author[Seoane]{J.B. Seoane-Sep\'{u}lveda}
\address{Departamento de An\'{a}lisis Matem\'{a}tico,\\
\indent Facultad de Ciencias Matem\'{a}ticas, \\
\indent Plaza de Ciencias 3, \\
\indent Universidad Complutense de Madrid,\\
\indent Madrid, 28040, Spain.\\
\indent
\textsc{\ and }\\
\indent
Instituto de Ciencias Matem\'aticas -- ICMAT \\
\indent Madrid, Spain.}
\email{jseoane@mat.ucm.es}
\thanks{D. Pellegrino and J.B. Seoane-Sep\'ulveda are supported by CNPq
Grant 401735/2013-3 - PVE - Linha 2.}

\begin{abstract}
H\"{o}lder's inequality, since its appearance in 1888, has played a
fundamental role in Mathematical Analysis and it is, without any doubt, one
of the milestones in Mathematics. It may seem strange that, nowadays, it
keeps resurfacing and bringing new insights to the mathematical community.
In this expository article we show how a variant of H\"{o}lder's inequality
(although well-known in PDEs) was essentially overlooked in Functional
Analysis and has had a crucial (and in some sense unexpected) influence in
very recent and major breakthroughs in Mathematics. Some of these recent
advances appeared in 2012-2014 and include the theory of Dirichlet series,
the famous Bohr radius problem, certain classical inequalities (such as
Bohnenblust--Hille or Hardy--Littlewood), or even Mathematical Physics.
\end{abstract}

\maketitle
\tableofcontents

\section{Introduction}

When Leonard James Rogers (1862-1933) and Otto H{\"{o}}lder (1859-1937)
discovered, independently, the famous inequality that (nowadays) holds H{%
\"{o}}lder's name (1889, \cite{holder1889}), they could have never imagined
that, at that precise moment, they had just started a \textquotedblleft
revolution\textquotedblright\ in Functional Analysis (we refer to \cite{Maligranda} for a detailed and historical exposition). This tool is a
fundamental inequality between integrals and an indispensable tool for the
study of, among others, $L_{p}$ spaces. Let us recall the classical $L_{p}$
version of this inequality.

\begin{theorem}[H{\"o}lder's inequality, 1889]
Let $(\Omega, \Sigma, \mu)$ be a measure space and let $p, q \in [1, \infty]$
with $1/p + 1/q = 1$ (H{\"o}lder's conjugates). Then, for all measurable
real or complex valued functions $f$ and $g$ on $\Omega$,
\begin{equation*}
\int |f g| d \mu \le \left( \int |f|^p d \mu \right)^{1/p} \left( \int |g|^q
d \mu \right)^{1/q}.
\end{equation*}
\end{theorem}

If one has $p,q\in (1,\infty )$, $f\in L_{p}(\mu )$, and $g\in L_{q}(\mu )$,
then this inequality becomes an equality if and only if $|f|^{p}$ and $%
|g|^{q}$ are linearly dependent in $L_{1}(\mu )$. When one has $p=q=2$ we
recover a form of the Cauchy-Schwarz inequality (or the
Cauchy-Bunyakovsky-Schwarz inequality). Also, H{\"{o}}lder's inequality is
used to prove Minkowski's inequality (the triangle inequality for $L_{p}$
spaces) and to establish that $L_{q}(\mu )$ is the dual space of $L_{p}(\mu
) $ for $p\in \lbrack 1,\infty )$. Of course, we are all familiar with these
classical applications of H{\"{o}}lder's inequality.

As it happens to almost every important result in mathematics, several
extensions and generalizations of it have appeared along the time; and in
the case of H{\"{o}}lder's inequality, this is not different. One of the
extensions is the variant of H{\"{o}}lder's inequality for mixed $L_{p}$
spaces. This inequality appeared in 1961, in the work of A. Benedek and R.
Panzone \cite{bene}. Mixed $L_{p}$ spaces may be seen as a pure exercise of
abstraction of the original notion of $L_{p}$ spaces, but as a matter of
fact we shall show that the theory developed in \cite{bene} plays a crucial
role in applications to quite different frameworks; it is intriguing that,
although widely known (the paper \cite{bene} has more than 100 citations,
mainly related to PDEs; we refer, for instance to \cite{veve, fff}) it was
overlooked in important fields of mathematics. This gap began to be filled
in 2012-2013, when H{\"{o}}lder's inequality for mixed $L_{p}$ spaces was
re-discovered as an interpolation-type result and we shall show that
different fields of Mathematics and even of Physics were positively
influenced. This expository paper is arranged as follows. Section 2 presents
some motivation to illustrate the subject of this article. Section 3 is
devoted to the aforementioned variant of H\"{o}lder's inequality (H\"{o}%
lder's inequality for mixed sums) with a short proof. This result was only
written in a proper and organized fashion in 1961 (\cite{bene}) but, as it
will be clear along this paper, at least in the topics gathered here
(Functional Analysis, Complex Analysis and Quantum Information Theory) it
was surely not been taken advantage of before 2012. Our approach is quite
different from the one employed in \cite{bene} and we shall follow the lines
of \cite{bps}. Section 4 will recall several useful inequalities that we
shall need and Section 5 is devoted to the recent applications of H\"{o}%
lder's inequality for mixed sums in Functional Analysis and Quantum
Information Theory, culminating with the solution of a classical problem
from Complex Analysis: the Bohr radius problem. Applications to the
improvement of the constants of the Hardy--Littlewood inequality and
separately summing operators are also given.

\section{Motivation: some interpolative puzzles}

\label{puzz}

As a motivation to the subject treated here, let us suppose that we have the
following two inequalities at hand, for certain complex scalar matrix $%
(a_{ij})_{i, j=1}^{N}$:

\begin{equation}
\sum\limits_{i=1}^{N}\left( \sum\limits_{j=1}^{N}\left\vert
a_{ij}\right\vert ^{2}\right) ^{\frac{1}{2}}\leq \mathrm{C}\text{ and }%
\sum\limits_{j=1}^{N}\left( \sum\limits_{i=1}^{N}\left\vert
a_{ij}\right\vert ^{2}\right) ^{\frac{1}{2}}\leq \mathrm{C}  \label{221}
\end{equation}%
for some constant $\mathrm{C}>0$ and all positive integers $N$.

How can one find an optimal exponent $r$ and a constant $\mathrm{C}_{1}>0$
such that%
\begin{equation*}
\left( \sum\limits_{i,j=1}^{N}\left\vert a_{ij}\right\vert ^{r}\right) ^{%
\frac{1}{r}}\leq \mathrm{C}_{1}
\end{equation*}%
for all positive integers $N?$ Moreover, how can one get a good (small)
constant $\mathrm{C}_{1}$?

This question (at least concerning the exponent $r$ can be solved in no less
than two ways: interpolation and H\"{o}lder's inequality).

First note that, by using a consequence of Minkowski's inequality (see \cite%
{garling}), we know that
\begin{equation}
\left( \sum\limits_{i=1}^{N}\left( \sum\limits_{j=1}^{N}\left\vert
a_{ij}\right\vert \right) ^{2}\right) ^{\frac{1}{2}}\leq
\sum\limits_{j=1}^{N}\left( \sum\limits_{i=1}^{N}\left\vert
a_{ij}\right\vert ^{2}\right) ^{\frac{1}{2}}\leq \mathrm{C}.  \label{222}
\end{equation}

If we use H\"{o}lder's inequality twice, we proceed as follows:
\begin{align*}
\sum\limits_{i,j=1}^{N}\left\vert a_{ij}\right\vert ^{\frac{4}{3}} = &
\sum\limits_{i=1}^{N}\left( \sum\limits_{j=1}^{N}\left\vert
a_{ij}\right\vert ^{\frac{2}{3}}\left\vert a_{ij}\right\vert ^{\frac{2}{3}%
}\right) \\
\leq & \sum\limits_{i=1}^{N}\left( \left( \sum\limits_{j=1}^{N}\left\vert
a_{ij}\right\vert ^{2}\right) ^{\frac{1}{3}}\left(
\sum\limits_{j=1}^{N}\left\vert a_{ij}\right\vert \right) ^{\frac{2}{3}%
}\right) \\
\leq & \left( \sum\limits_{i=1}^{N}\left( \sum\limits_{j=1}^{N}\left\vert
a_{ij}\right\vert ^{2}\right) ^{\frac{1}{2}}\right) ^{\frac{2}{3}}\left(
\sum\limits_{i=1}^{N}\left( \sum\limits_{j=1}^{N}\left\vert
a_{ij}\right\vert \right) ^{2}\right) ^{\frac{1}{3}} \\
= & \left[ \sum\limits_{i=1}^{N}\left( \sum\limits_{j=1}^{N}\left\vert
a_{ij}\right\vert ^{2}\right) ^{\frac{1}{2}}\right] ^{\frac{2}{3}}\left[
\left( \sum\limits_{i=1}^{N}\left( \sum\limits_{j=1}^{N}\left\vert
a_{ij}\right\vert \right) ^{2}\right) ^{\frac{1}{2}}\right] ^{\frac{2}{3}}
\leq \mathrm{C}^{\frac{4}{3}}.
\end{align*}

By complex interpolation (see \cite{berg.lofst}) the solution is shorter;
essentially we have two mixed inequalities with exponents $\left( 1,2\right)
$ in equation \eqref{221} and $\left( 2,1\right) $ in equation \eqref{222}.
By interpolating these exponents with $\theta _{1}=\theta _{2}=1/2$ we
obtain an exponent $\left( 4/3,4/3\right) $ with constant $\mathrm{C}$.

The optimality of the exponent $4/3$ can be proved using the
Kahane--Salem--Zygmund inequality (Theorem \ref{Littlewood43}).

The use of H\"{o}lder's inequality as above becomes almost impossible if we
have, for instance, 10 sums, 100 sums, etc. The reader can test the case of
three sums using H\"{o}lder's inequality. More precisely, as a simple
illustration suppose that

\begin{equation*}
\sum\limits_{\sigma (i)=1}^{N}\left( \sum\limits_{\sigma
(j)=1}^{N}\sum\limits_{\sigma (k)=1}^{N}\left\vert a_{ijk}\right\vert
^{2}\right) ^{\frac{1}{2}}\leq \mathrm{C}
\end{equation*}%
for all bijections $\sigma :\left\{ i,j,k\right\} \rightarrow \left\{
i,j,k\right\} $ and all $N.$ How can we find an optimal exponent $r$ and a
constant $\mathrm{C}_{1}$ such that%
\begin{equation*}
\left( \sum\limits_{i,j,k=1}^{N}\left\vert a_{ijk}\right\vert ^{r}\right) ^{%
\frac{1}{r}}\leq \mathrm{C}_{1}
\end{equation*}
for every $N$?

The search for good constants dominating the respective inequalities is
highly important for applications (see Section 5) and has an extra
ingredient when we are using the interpolative approach: the main point is
that different interpolations may result in the same exponent, but the
constants involved differ. Thus, we must investigate what exponents we shall
use to interpolate. More precisely, as we will see in Section 5, the
Bohnenblust--Hille inequality for $3$-linear forms asserts that there is a
constant $\mathrm{C}_{3}\geq 1$ such that, for all $3$--linear forms $T:\ell
_{\infty}^{N}\times \ell _{\infty }^{N}\times \ell _{\infty }^{N}\rightarrow
\mathbb{K}$,
\begin{equation*}
\left( \sum\limits_{i_{1},i_{2},i_{3}=1}^{N}\left\vert
T(e_{i_{^{1}}},e_{i_{2}},e_{i_{3}})\right\vert ^{\frac{3}{2}}\right) ^{\frac{%
2}{3}}\leq \mathrm{C}_{3}\left\Vert T\right\Vert ,
\end{equation*}
and all $N$, where
\begin{equation*}
\Vert T\Vert := \sup_{\left\Vert z^{(1)}\right\Vert=1,\dots,\left\Vert
z^{(m)}\right\Vert =1} \left\vert T\left( z^{\left(1\right) },\ldots
,z^{\left( m\right)}\right) \right\vert \quad \text{ for all } m \in \mathbb{%
N}.
\end{equation*}
However, the exponent $3/2$ can be obtained by a \textquotedblleft
multiple\textquotedblright\ interpolation of exponents of inequalities of
the form
\begin{equation*}
\left( \sum\limits_{i_{1}=1}^{N} \left( \sum\limits_{i_{2}=1}^{N}\left(
\sum\limits_{i_{3}=1}^{N}\left\vert
T(e_{i_{^{1}}},e_{i_{2}},e_{i_{3}})\right\vert ^{q_{3}}\right) ^{\frac{q_{2}%
}{q_{3}}}\right) ^{\frac{q_{1}}{q_{2}}}\right) ^{\frac{1}{q_{1}}}\leq
\mathrm{C}\left\Vert T\right\Vert ,
\end{equation*}
with
\begin{equation*}
(q_{1},q_{2},q_{3})=(1,2,2),\left( 2,1,2\right) \text{ and }\left(
2,2,1\right)
\end{equation*}
or
\begin{equation*}
(q_{1},q_{2},q_{3})= \left( \frac{4}{3},\frac{4}{3},2 \right),\left( \frac{4%
}{3},2,\frac{4}{3}\right) \text{ and } \left( 2,\frac{4}{3},\frac{4}{3}%
\right)
\end{equation*}
and the last procedure provides quite better constants. This is a simple
illustration of the core of the new advances that lead to the results
reported in this expository paper.

The theory of $L_{p}$ spaces with mixed norms seems to have been created in
1961, \cite{bene}, including a H\"{o}lder inequality in this framework.
However, as we describe along this paper the full strength of this
inequality was overlooked and very recently important advances in different
fields of Mathematics and Physics were achieved with the help of this H\"{o}%
lder inequality (also called interpolative approach).

\section{H\"{o}lder's inequality revisited\label{sec3}}

Essentially, the simplest version of the H\"{o}lder inequality asserts that
if $1/p+1/q=1$ and $\left( a_{j}\right) \in \ell _{p},$ $\left( b_{j}\right)
\in \ell _{q}$ then $\left( a_{j}b_{j}\right) \in \ell _{1}.$ In this
section we present a variation of this result, which was apparently
overlooked in Functional Analysis (but not in PDEs) in the last decades.
This variant is a key result of a number of important recent advances in
Mathematical Analysis and Mathematical Physics that appeared in the last few
years.

The previous result may have been seen as a variant of the following general
H\"{o}lder's inequality presented in the classical paper \cite{bene} on
mixed norms in $L_{p}$ spaces. We shall now work with $L_p(\mathbb{N}) =
\ell_p$, since it is the case we are interested in. We need to recall some
useful multi-index notation: for a positive integer $m$ and a non-void
subset $D\subset \mathbb{N}$ we denote the set of multi-indices $\mathbf{i}%
=(i_{1},\dots ,i_{m})$, with each $i_{k}\in D$, by
\begin{equation*}
\mathcal{M}(m,D):=\left\{ \mathbf{i}=(i_{1},\dots ,i_{m})\in \mathbb{N}%
^{m};\ i_{k}\in D,\,k=1,\dots ,m\right\} =D^{m}.
\end{equation*}
We also denote
\begin{equation*}
\mathcal{M}(m,n):=\mathcal{M}(m,\{1, 2, \ldots, n\}).
\end{equation*}

For $\mathbf{p}=(p_{1},\dots ,p_{m})\in \lbrack 1, \infty )^{m}$, and a
Banach space $X$, let us consider the space
\begin{equation*}
\ell _{\mathbf{p}}(X):=\ell _{p_{1}}\left( \ell _{p_{2}}\left( \dots \left(
\ell _{p_{m}}(X)\right) \dots \right) \right) ,
\end{equation*}%
namely, a vector matrix $\left( x_{\mathbf{i}}\right) _{\mathbf{i}\in
\mathcal{M}(m,\mathbb{N})}\in \ell _{\mathbf{p}}(X)$ if, and only if,
\begin{equation*}
\left( \sum_{i_{1}=1}^{\infty }\left( \sum_{i_{2}=1}^{\infty }\left( \dots
\left( \sum_{i_{m-1}=1}^{\infty }\left( \sum_{i_{m}=1}^{\infty }\left\Vert
x_{\mathbf{i}}\right\Vert _{X}^{p_{m}}\right) ^{\frac{p_{m-1}}{p_{m}}%
}\right) ^{\frac{p_{m-2}}{p_{m-1}}}\dots \right) ^{\frac{p_{2}}{p_{3}}%
}\right) ^{\frac{p_{1}}{p_{2}}}\right) ^{\frac{1}{p_{1}}}< \infty .
\end{equation*}%
When $X=\mathbb{K}$, we just write $\ell _{\mathbf{p}}$ instead of $\ell _{%
\mathbf{p}}(\mathbb{K})$.

Also, we deal with the coordinatewise product of two scalar matrices $%
\mathbf{a}=\left( a_{\mathbf{i}}\right) _{\mathbf{i}\in \mathcal{M}(m,n)}$
and $\mathbf{b}=\left( b_{\mathbf{i}}\right) _{\mathbf{i}\in \mathcal{M}%
(m,n)}$, \emph{i.e.},
\begin{equation*}
\mathbf{a}\mathbf{b}:=\left( a_{\mathbf{i}}b_{\mathbf{i}}\right) _{\mathbf{i}%
\in \mathcal{M}(m,n)}.
\end{equation*}

The following result seems to be first observed by A. Benedek and R. Panzone
(see \cite{bene}):

\begin{theorem}[H\"{o}lder's inequality for mixed $\ell _{\mathbf{p}}$ spaces%
]
\label{GenHol} Let $\mathbf{r},\mathbf{q}(1),\dots ,\mathbf{q}(N)\in \lbrack
1, \infty ]^{m}$ such that
\begin{equation*}
\frac{1}{r_{j}}=\frac{1}{q_{j}(1)}+\cdots +\frac{1}{q_{j}(N)}, \quad j \in
\{1, 2, \ldots, m\}
\end{equation*}%
and let $\mathbf{a}_{k},\,k=1,\dots ,N$ be scalar $m$-square matrix. Then
\begin{equation*}
\left\Vert \prod_{k=1}^{N}\mathbf{a}_{k}\right\Vert _{\mathbf{r}}\leq
\prod_{k=1}^{N}\left\Vert \mathbf{a}_{k}\right\Vert _{\mathbf{q}(k)}.
\end{equation*}
\end{theorem}

Remember that, the previous inequality means the following:

\begin{quote}
{\small $\displaystyle
\left( \sum_{i_{1}=1}^{n}\left( \dots \left( \sum_{i_{m}=1}^{n} |a_{\mathbf{i%
}}^1 \cdot a_{\mathbf{i}}^2 \cdot \ldots \cdot a_{\mathbf{i}%
}^N|^{q_{m}}\right) ^{\frac{q_{m-1}}{q_{m}}}\dots \right) ^{\frac{q_{1}}{%
q_{2}}}\right) ^{\frac{1}{q_{1}}}$\newline
$\displaystyle \leq \prod_{k=1}^{N}\left[ \left( \sum_{i_{1}=1}^{n}\left(
\dots \left( \sum_{i_{m}=1}^{n}|a_{\mathbf{i}}|^{q_{m}(k)}\right) ^{\frac{%
q_{m-1}(k)}{q_{m}(k)}}\dots \right) ^{\frac{q_{1}(k)}{q_{2}(k)}}\right) ^{%
\frac{1}{q_{1}(k)}}\right], $ }
\end{quote}

Using the above result we are able to recover the interpolative inequality
from \cites{alb, bps, n, joedson} (Theorem \ref{gen.interp} below), that we
can also, in some sense, call H\"{o}lder's inequality for multiple
exponents. Under the point of view of interpolation theory it is not a
complicated result but, just in 2013, it began to be used in all its full
strength. Its applications (in different fields) are impressive, as we shall
illustrate in the remaining of the paper. Just before that, for a positive
real number $\theta $, let us define $\mathbf{a}^{\theta}:=\left( a_{\mathbf{%
i}}^{\theta }\right) _{\mathbf{i}\in \mathcal{M}(m,n)}$. It is
straightforward to see that
\begin{equation*}
\left\Vert \mathbf{a}^{\theta }\right\Vert _{\mathbf{q}/\theta }=\left\Vert
\mathbf{a}\right\Vert _{\mathbf{q}}^{\theta },
\end{equation*}%
where $\mathbf{q}/\theta :=\left( q_{1}/\theta ,\dots ,q_{m}/\theta \right) $%
.

\begin{theorem}[H\"{o}lder's inequality for multiple exponents
-interpolative approach-]
\label{gen.interp}  Let $m,n,N$ be positive integers and $\mathbf{q},\mathbf{%
q}(1),\ldots ,\mathbf{q}(N)\in \lbrack 1, \infty )^{m}$ be such that $\left(
\frac{1}{q_{1}},\dots ,\frac{1}{q_{m}}\right) $ belongs to the convex hull
of $\left( \frac{1}{q_{1}(k)},\dots ,\frac{1}{q_{m}(k)}\right) $, $%
k=1,\ldots ,N$. Then for all scalar matrix $\mathbf{a}=\left( a_{\mathbf{i}%
}\right) _{\mathbf{i}\in \mathcal{M}(m,n)}$,
\begin{equation*}
\left\Vert \mathbf{a}\right\Vert _{\mathbf{q}}\leq \prod_{k=1}^{N}\left\Vert
\mathbf{a}\right\Vert _{\mathbf{q}(k)}^{\theta _{k}},
\end{equation*}%
\emph{i.e.},

{\small $\displaystyle
\left( \sum_{i_{1}=1}^{n}\left( \dots \left( \sum_{i_{m}=1}^{n}|a_{\mathbf{i}%
}|^{q_{m}}\right) ^{\frac{q_{m-1}}{q_{m}}}\dots \right) ^{\frac{q_{1}}{q_{2}}%
}\right) ^{\frac{1}{q_{1}}}$\newline
$\displaystyle\leq \prod_{k=1}^{N}\left[ \left(\sum_{i_{1}=1}^{n}\left(
\dots \left( \sum_{i_{m}=1}^{n}|a_{\mathbf{i}}|^{q_{m}(k)}\right)^{\frac{%
q_{m-1}(k)}{q_{m}(k)}}\dots \right) ^{\frac{q_{1}(k)}{q_{2}(k)}}\right) ^{%
\frac{1}{q_{1}(k)}}\right] ^{\theta _{k}}, $ }

where $\theta _{k}$ are the coordinates of $\left( \frac{1}{q_{1}(k)},\dots ,%
\frac{1}{q_{m}(k)}\right) $ on the convex hull.
\end{theorem}

\begin{proof}
For $j= 1, \ldots, m$ we have
\begin{equation*}
\frac{1}{q_j} = \frac{\theta_1}{q_j(1)}+ \ldots + \frac{\theta_N}{q_j(N)} =
\frac{1}{q_j(1)/\theta_1}+ \ldots + \frac{1}{q_j(N)/\theta_N}.
\end{equation*}

Since $\left\Vert \mathbf{a}^{\theta _{k}}\right\Vert _{\mathbf{q}(k)/\theta
_{k}}=\left\Vert \mathbf{a}\right\Vert _{\mathbf{q}(k)}^{\theta _{k}}$, by
the H\"{o}lder inequality for mixed $\ell _{\mathbf{p}}$ spaces we conlude
that
\begin{equation*}
\left\Vert \mathbf{a}\right\Vert _{\mathbf{q}}=\left\Vert \mathbf{a}^{\theta
_{1}+\cdots +\theta _{N}}\right\Vert _{\mathbf{q}}=\left\Vert \prod_{k=1}^{N}%
\mathbf{a}^{\theta _{k}}\right\Vert _{\mathbf{q}}\leq
\prod_{k=1}^{N}\left\Vert \mathbf{a}^{\theta _{k}}\right\Vert _{\mathbf{q}%
(k)/\theta _{k}}=\prod_{k=1}^{N}\left\Vert \mathbf{a}\right\Vert _{\mathbf{q}%
(k)}^{\theta _{k}}.
\end{equation*}
\end{proof}

For the sake of completeness of this article, we would also like to present
the following proof, which is based on interpolation.

\begin{proof}[(Interpolative Approach)]
We just follow the lines of \cite[Proposition 2.1]{alb}. Proceeding by
induction on $N$ and using that, for any Banach space $X$ and $\theta \in
[0,1]$,
\begin{equation*}
\ell _{\mathbf{r}}(X)=\left[ \ell _{\mathbf{p}}(X),\ell _{\mathbf{q}}(X)%
\right] _{\theta },
\end{equation*}%
with $\frac{1}{r_{i}}=\frac{\theta }{p_{i}}+\frac{1-\theta }{q_{i}}$, for $%
i=1,\dots ,m$ (see \cite{berg.lofst}). If
\begin{equation*}
\frac{1}{q_{i}}=\frac{\theta _{1}}{q_{i}(1)}+\cdots +\frac{\theta _{N}}{%
q_{i}(N)},
\end{equation*}
with $\sum_{k=1}^{N}\theta _{k}=1$ and each $\theta _{k}\in \lbrack 0,1]$,
then we also have
\begin{equation*}
\frac{1}{q_{i}}=\frac{\theta _{1}}{q_{i}(1)}+\frac{1-\theta _{1}}{p_{i}},
\end{equation*}%
setting
\begin{equation*}
\frac{1}{p_{i}}=\frac{\alpha _{2}}{q_{i}(2)}+\cdots +\frac{\alpha _{N}}{%
q_{i}(N)},\ \ \ \mbox{and}\ \alpha _{j}=\frac{\theta _{j}}{1-\theta _{1}},
\end{equation*}%
for $i=1,\dots ,m$ and $j=2,\dots ,N$. So $\alpha _{j}\in \lbrack 0,1]$ and $%
\sum_{j=2}^{N}\alpha _{j}=1$. Therefore, by the induction hypothesis, we
conclude that
\begin{equation*}
\left\Vert \mathbf{a}\right\Vert _{\mathbf{q}}\leq \left\Vert \mathbf{a}%
\right\Vert _{\mathbf{q}(1)}^{\theta _{1}}\cdot \left\Vert \mathbf{a}%
\right\Vert _{\mathbf{p}}^{1-\theta _{1}}\leq \left\Vert \mathbf{a}%
\right\Vert _{\mathbf{q}(1)}^{\theta _{1}}\cdot \left[ \prod_{j=2}^{N}\left%
\Vert \mathbf{a}\right\Vert _{\mathbf{q}(j)}^{\alpha _{j}}\right] ^{1-\theta
_{1}}=\prod_{k=1}^{N}\left\Vert \mathbf{a}\right\Vert _{\mathbf{q}%
(k)}^{\theta _{k}}.
\end{equation*}
\end{proof}

Combining the previous result with Minkowski's inequality we have a very
useful inequality (see \cite[Remark 2.2]{bps}):

\begin{corollary}
\label{cor33} \label{blei.interp} Let $m,n$ be positive integers, $1\leq
k\leq m$ and $1\leq s\leq q$. Then for all scalar matrix $\left( a_{\mathbf{i%
}}\right) _{\mathbf{i}\in \mathcal{M}(m,n)}$,
\begin{equation*}
\left( \sum_{\mathbf{i}\in \mathcal{M}(m,n)}\left\vert a_{\mathbf{i}%
}\right\vert ^{\rho }\right) ^{\frac{1}{\rho }}\leq \prod_{S\in \mathcal{P}%
_{k}\left( m\right) }\left( \sum_{\mathbf{i}_{S}}\left( \sum_{\mathbf{i}_{%
\widehat{S}}}\left\vert a_{\mathbf{i}}\right\vert ^{q}\right) ^{\frac{s}{q}%
}\right) ^{\frac{1}{s}\cdot \frac{1}{\binom{m}{k}}},
\end{equation*}%
where
\begin{equation*}
\rho :=\frac{msq}{kq+(m-k)s}
\end{equation*}
and $\mathcal{P}_{k}\left( m\right)$ stands for the set of subsets $S
\subseteq \{1, \ldots, m\}$ with card($S$)$=k$.
\end{corollary}

The above corollary shows that Blei's inequality (see Corollary \ref{b}
below) is just a very particular case of a huge family of similar
inequalities. For our purposes, the crucial point is that the use of Blei's
inequality is far from being a good option to obtain good estimates for the
constants of the Bohnenblust--Hille and related inequalities. Just to
illustrate the strength of Theorem \ref{gen.interp} and Corollary \ref{cor33}%
, we present here a very simple proof (see \cite{bps}) of Blei's inequality.

\begin{corollary}[Blei's inequality - Defant, Popa, Schwarting approach,
\protect\cite{defa}]
\label{b} Let $A$ and $B$ be two finite non-void index sets. Let $%
(a_{ij})_{(i,j)\in A\times B}$ be a scalar matrix with positive entries, and
denote its columns by $\alpha _{j}=(a_{ij})_{i\in A}$ and its rows by $\beta
_{i}=(a_{ij})_{j\in B}.$ Then, for $q,s_{1},s_{2}\geq 1$ with $q>\max
(s_{1},s_{2})$ we have
\begin{equation*}
\left( \sum_{(i,j)\in A\times B}a_{ij}^{w(s_{1},s_{2})}\right) ^{\frac{1}{%
w(s_{1},s_{2})}}\leq \left( \sum_{i\in A}\left\Vert \beta _{i}\right\Vert
_{q}^{s_{1}}\right) ^{\frac{f(s_{1},s_{2})}{s_{1}}}\left( \sum_{j\in
B}\left\Vert \alpha _{j}\right\Vert _{q}^{s_{2}}\right) ^{\frac{%
f(s_{2},s_{1})}{s_{2}}},
\end{equation*}
with
\begin{align*}
w& :[1,q)^{2}\rightarrow \lbrack 0,\infty ),\text{ }w(x,y):=\frac{%
q^{2}(x+y)-2qxy}{q^{2}-xy}, \\
f& :[1,q)^{2}\rightarrow \lbrack 0,\infty ),\text{ }f(x,y):=\frac{q^{2}x-qxy%
}{q^{2}(x+y)-2qxy}.
\end{align*}
\end{corollary}

\begin{proof}
Let us consider the exponents
\begin{equation*}
\left( q,s_{2}\right) ,\left( s_{1},q\right)
\end{equation*}%
and
\begin{equation*}
\left( \theta _{1},\theta _{2}\right) =\left(
f(s_{2},s_{1}),f(s_{1},s_{2})\right) .
\end{equation*}%
Note that $(w\left( s_{1},s_{2}\right) ,w\left( s_{1},s_{2}\right) )$ is
obtained by interpolating $\left( q,s_{2}\right) $ and $\left(
s_{1},q\right) $ with $\theta _{1},\theta _{2}$, respectively. Then, from
Theorem \ref{gen.interp}, we have
\begin{equation*}
\left( \sum_{(i,j)\in A\times B}a_{ij}^{w(s_{1},s_{2})}\right) ^{\frac{1}{%
w(s_{1},s_{2})}}\leq \left( \sum_{i\in A}\left\Vert \beta _{i}\right\Vert
_{q}^{s_{1}}\right) ^{\frac{f(s_{1},s_{2})}{s_{1}}}\left( \sum_{i\in
A}\left\Vert \beta _{i}\right\Vert _{s_{2}}^{q}\right) ^{\frac{f(s_{2},s_{1})%
}{q}}.
\end{equation*}%
Now, since $q>s_{2}$ we just need to use Propositon \ref{8776} to change the
order of the last sum.
\end{proof}

We invite the interest reader to compare the above proof with the proof
presented in \cite[pages 226-227]{defa}, in which the classical H\"{o}lder's
inequality is needed several times.

\section{Some useful inequalities\label{445}}

The main recent advances presented here are direct or indirect consequence
of the improvements obtained in the polynomial and multilinear
Bohnenblust--Hille inequalities (and these improvements were obtained by
using the theory of mixed $L_{p}$ spaces and the core of the results lie in
the variant of H\"{o}lder's inequality (Theorem \ref{gen.interp}). However
we also need three other important ingredients: the Khinchine inequality
(and its version for multiple sums), Kahane--Salem--Zygmund's inequality in
its polynomial and multilinear versions and a variant of Minkowski's
inequality. Before that, let us provide a brief account on polynomials and
multilinear operators, that shall be needed in the remaining of the article.

Polynomials in Banach spaces (at least for complex scalars) are of
fundamental importance in the theory of Infinite Dimensional Holomorphy (see
\cite{dineen, mujica}). In general the theory of polynomials and multilinear
operators between normed spaces has its importance in different areas of
Mathematics, from Number Theory, or Dirichlet series, to Functional Analysis.

In this section we recall the concepts of polynomials and multilinear
operators between Banach spaces and some ``folkloric results'', that will be
needed here.

If $E_{1},\ldots,E_{m}$, and $F$ are vector spaces, a $m$-linear operator $%
T:E_{1}\times \cdots \times E_{m}\rightarrow F$ is a map that is linear in
each coordinate separately. When $E_{1}=\cdots =E_{m}=E$ we say that $T$ is
symmetric if $T(x_{\sigma (1)}, \ldots ,x_{\sigma (m)})=T(x_{1}, \ldots
,x_{m})$ for all bijections $\sigma :\left\{ 1, \ldots ,m\right\}
\rightarrow \left\{ 1, \ldots ,m\right\} .$

If $E,F$ are vector spaces, a $m$-homogeneous polynomial is a map $%
P:E\rightarrow F$ such that%
\begin{equation*}
P(x)=T(x, \ldots ,x)
\end{equation*}%
for some $m$-linear operator $T:E\times \cdots \times E\rightarrow F.$
Continuity is defined is the obvious fashion.

Fixed $E,F,$ $E_{1}, \ldots ,E_{m}$, the spaces of continuous $m$%
-homogeneous polynomials from $E$ to $F$ are represented by $\mathcal{P}%
\left( ^{m}E;F\right) $ and the space of continuous multilinear operators
from $E_{1}\times \cdots \times E_{m}$ to $F$ is denoted by $\mathcal{L}%
\left( E_{1}, \ldots ,E_{m};F\right) .$ Both vector spaces are Banach spaces
when endowed with the $\sup $ norm in the unit ball of $B_{E}$ or in product
of the the unit balls $B_{E_{1}}\times \cdots \times B_{E_{m}}.$

The following characterizations of continuous polynomials are elementary
(analogous results holds for multilinear operators):

\begin{proposition}
\label{20}Let $E,F$ be vector spaces, $P\in \mathcal{P}\left(
^{m}E;F\right). $ The following assertions are equivalent:

\begin{itemize}
\item[(i)] $P\in \mathcal{P}\left( ^{m}E;F\right) $;

\item[(ii)] $P$ is continuous at zero;

\item[(iii)] There is a constant $\mathrm{M}>0$ such that $\left\Vert
P\left( x\right) \right\Vert \leq \mathrm{M}\left\Vert x\right\Vert ^{m}$,
for all $x\in E$;
\end{itemize}
\end{proposition}

The Polarization Formula relates polynomials and symmetric multilinear
operators in a very useful way. Its proof is a kind of consequence of the
Leibnitz formula and some \emph{combinatorial tricks} (see \cite{dineen,
mujica}).

\begin{theorem}[Polarization Formula]
\label{FPFEV}Let $E,F$ be linear spaces. If $T\in \mathcal{L}(^{m}E;F)$ is
symmetric then%
\begin{equation*}
T(x_{1}, \ldots ,x_{m})=\frac{1}{m!2^{m}}\underset{\varepsilon _{i}=\pm 1}{%
\sum }\varepsilon _{1}\cdots \varepsilon _{m}T(x_{0}+\varepsilon
_{1}x_{1}+\cdots +\varepsilon _{m}x_{m})^{m},
\end{equation*}%
for all $x_{0},x_{1},x_{2}, \ldots ,x_{m}\in E.$
\end{theorem}

The following result is an immediate consequence of the Polarization Formula:

\begin{corollary}
\label{corpf}For each $m$-homogeneous polynomial there is a unique $m$%
-linear operator associated to it. In other words, if $P$ is a $m$%
-homogeneous polynomial, then there exists only one symmetric $m$-linear
operator $T$ (sometimes called \textit{polar} of $P$) such that
\begin{equation*}
P(x)=T(x, \ldots ,x)
\end{equation*}
for all $x.$
\end{corollary}

In general, if $T$ is the symmetric $m$-linear operator associated to a $m$%
-homogeneous polynomial $P$ we have%
\begin{equation}
\left\Vert P\right\Vert \leq \frac{m^{m}}{m!}\left\Vert T\right\Vert,
\label{polar}
\end{equation}
where $\Vert P\Vert=\sup_{\left\Vert z\right\Vert =1}|P(z)|$. The constant $%
\frac{m^m}{m!}$ is usually called polarization constant.

If $P$ is a homogeneous polynomial of degree $m$ on ${\mathbb{K}}^{n}$ given
by
\begin{equation*}
P(x_{1},\ldots ,x_{n})=\sum_{|\alpha |=m}a_{\alpha }\mathbf{{x}^{\alpha },}
\end{equation*}%
and $L$ is the polar of $P$, then
\begin{equation}
L(e_{1}^{\alpha _{1}},\ldots ,e_{n}^{\alpha _{n}})=\frac{a_{\alpha }}{\binom{%
m}{\alpha }},  \label{800}
\end{equation}%
where $\{e_{1},\ldots ,e_{n}\}$ is the canonical basis of ${\mathbb{K}}^{n}$
and $e_{k}^{\alpha_{k}}$ stands for $e_{k}$ repeated $\alpha_{k}$ times, the
$\alpha_j$'s are non negative integers with $\sum_{j=1}^{n} \alpha_j = \alpha
$, and $\mathbf{{x}^{\alpha}} = x_{1}^{\alpha_1} \cdot \ldots \cdot
x_{n}^{\alpha_n}$.

\subsection{Khinchine's inequality}

The Khinchine inequality in its modern presentation has its roots in \cite%
{zyg}. Let $(\varepsilon _{i})_{i\geq 1}$ be a sequence of independent
Rademacher variables. Then, for any $p\in \lbrack 1,2]$, there exists a
constant $\mathrm{A}_{\mathbb{R},p}$ such that, for any sequence $(a_{i})$
of real numbers with finite support,
\begin{equation*}
\left( \sum_{i=1}^{ \infty }|a_{i}|^{2}\right) ^{1/2}\leq \mathrm{A}_{%
\mathbb{R},p}^{-1}\left( \int_{[0,1]^m }\left\vert \sum_{i=1}^{ \infty
}a_{i}\varepsilon _{i}(\omega )\right\vert ^{p}d\omega \right) ^{1/p}.
\end{equation*}%
For complex scalars it more useful (since it gives better constants) to use
the following version of Khinchine's inequality (called Khinchine's
inequality with Steinhaus variables): for any $p\in \lbrack 1,2]$, there
exists a constant $\mathrm{A}_{\mathbb{C},p}$ such that, for any sequence $%
(a_{i})$ of complex numbers with finite support
\begin{equation*}
\left( \sum_{i=1}^{ \infty }|a_{i}|^{2}\right) ^{1/2}\leq \mathrm{A}_{%
\mathbb{C},p}^{-1}\left( \int_{\mathbb{T}^{\infty }}\left\vert \sum_{i=1}^{
\infty }a_{i}z_{i}\right\vert ^{p}dz\right) ^{1/p},
\end{equation*}%
with $\mathbb{T}^{\infty }$ denoting the infinite polycircle, i.e.,

\begin{equation*}
\mathbb{T}^{\infty }=\left\{ z=\left( z_{i}\right) _{i\in \mathbb{N}}\in
\mathbb{C}^{\mathbb{N}}:\left\vert z_{i}\right\vert =1\text{ for all }i\in
\mathbb{N}\right\},
\end{equation*}%
and $dz$ denoting the standard Lebesgue probability measure on $\mathbb{T}%
^{\infty }$. The best constants $\mathrm{A}_{\mathbb{R},p}$ and $\mathrm{A}_{%
\mathbb{C},p}$ were obtained by Haagerup and K\"{o}nig, respectively (see
\cite{Haa} and \cite{KKw}). More precisely,

\begin{itemize}
\item $\mathrm{A}_{\mathbb{R},p}=\frac{1}{\sqrt{2}}\left( \frac{\Gamma
\left( \frac{1+p}{2}\right) }{\sqrt{\pi }}\right) ^{1/p}$ if $p>p_{0}\approx
1.8474$;

\item $\mathrm{A}_{\mathbb{R},p}=2^{\frac{1}{2}-\frac{1}{p}}$ if $p<p_{0}$;

\item $\mathrm{A}_{\mathbb{C},p}=\Gamma \left( \frac{p+2}{2}\right) ^{1/p}$
if $p\in \lbrack 1,2]$.
\end{itemize}

The (apparently) \emph{strange} value $p_{0}\approx 1.8474$ is, to be
precise, the unique number $p_{0}\in (1,2)$ with
\begin{equation*}
\Gamma \left( \frac{p_{0}+1}{2}\right) =\frac{\sqrt{\pi }}{2}.
\end{equation*}%
The notation $\mathrm{A}_{\mathbb{K},p}$ will be kept along this paper.

Using Fubini's theorem and Minkowski's inequality (see, for instance, \cite[%
Lemma 2.2]{defa} for the real case and \cite[Theorem 2.2]{ddss} for the
complex case), these inequalities have a multilinear version: for any $%
n,m\geq 1$, for any family $(a_{\mathbf{i}})_{\mathbf{i}\in \mathbb{N}^{m}}$
of real (resp. complex) numbers with finite support,
\begin{equation*}
\left( \sum_{\mathbf{i}\in \mathbb{N}^{m}}|a_{\mathbf{i}}|^{2}\right)
^{1/2}\leq \mathrm{A}_{\mathbb{R},p}^{-m}\left( \int_{[0,1]^m }\left\vert
\sum_{\mathbf{i}\in \mathbb{N}^{m}}a_{\mathbf{i}}\varepsilon
_{i_{1}}^{(1)}(\omega_1 )\dots \varepsilon _{i_{m}}^{(m)}(\omega_m
)\right\vert ^{p}d\omega_1 \cdots d\omega_m \right) ^{1/p}
\end{equation*}%
where $(\varepsilon _{i}^{(1)}),\dots ,(\varepsilon _{i}^{(m)})$ are
sequences of independent Rademacher variables (resp.
\begin{equation*}
\left( \sum_{\mathbf{i}\in \mathbb{N}^{m}}|a_{\mathbf{i}}|^{2}\right)
^{1/2}\leq \mathrm{A}_{\mathbb{C},p}^{-m}\left( \int_{(\mathbb{T}^{\infty
})^{m}}\left\vert \sum_{\mathbf{i}\in \mathbb{N}^{m}}a_{\mathbf{i}%
}z_{i_{1}}^{(1)}\dots z_{i_{m}}^{(m)}\right\vert ^{p}dz^{(1)}\dots
dz^{(m)}\right) ^{1/p},
\end{equation*}%
in the complex case).

\subsection{Kahane--Salem--Zygmund's inequality}

The essence of the Kahane--Salem--Zygmund inequalities, as we describe
below, probably appeared for the first time in \cite{Kahane}, but our
approach follows the lines of Boas' paper \cite{korea}. Paraphrasing Boas,
the Kahane--Salem--Zygmund inequalities use probabilistic methods to
construct a homogeneous polynomial (or multilinear operator) with a
relatively small supremum norm but relatively large majorant function. Both
the multilinear and polynomial versions are needed for our goals.

\begin{theorem}[Kahane--Salem--Zygmund's inequality - Multilinear version,
\protect\cite{korea}]
\label{novv1} Let $m,n$ be positive integers. There exists a $m$-linear map $%
\displaystyle T_{m,n}:\ell _{\infty }^{n}\times \cdots \times \ell
_{\infty}^{n}\rightarrow \mathbb{K}$ of the form
\begin{equation*}
\displaystyle T_{m,n}(z^{(1)}, \ldots, z^{(m)})=\sum_{i_{1}, \ldots, i_{m} =
1}^{n}\pm z_{i_{1}}^{(1)} \ldots z_{i_{m}}^{(m)}
\end{equation*}
such that
\begin{equation*}
\left\Vert T_{m,n}\right\Vert \leq \sqrt{32\log \left( 6m\right) }\times n^{%
\frac{m+1}{2}}\times \sqrt{n!}.
\end{equation*}
\end{theorem}

The original version of the Kahane--Salem--Zygmund appears in the framework
of complex scalars but it is simple to verify that the same result (with the
same constants) holds for real scalars. The folowing result is corollary of
the previous, now for polynomials, and it will also be important for our
aims.

\begin{theorem}[Kahane--Salem--Zygmund's inequality - Polynomial version,
\protect\cite{korea}]
\label{novv2}Let $m,n$ be positive integers. Then there exists a $m$%
-homogeneous polynomial $P:\ell _{\infty }^{n}\rightarrow \mathbb{K}$ of the
form
\begin{equation*}
P_{m,n}(\mathbf{z})={\textstyle}\sum_{|\alpha |=d}\pm \binom{m}{\alpha }%
\mathbf{z}^{\alpha }
\end{equation*}%
such that
\begin{equation*}
\left\Vert P_{m,n}\right\Vert \leq \sqrt{32\log \left( 6m\right) }\times n^{%
\frac{m+1}{2}}\times \sqrt{n!}.
\end{equation*}
\end{theorem}

\subsection{A corollary to Minkowski's inequality}

Minkowski's inequality is a very well-known result that helps to prove that $%
L_{p}$ spaces are Banach spaces: it is the triangle inequality for $L_{p}$
spaces. We need a somewhat well known result, which is a corollary of one of
the many versions of Minkowski's inequality, whose proof can be found, for
instance, in \cite{garling}.

\begin{proposition}[Corollary to Minkowski's inequality]
\label{8776}For any $0<p\leq q< \infty $ and for any matrix of complex
numbers $(c_{ij})_{i,j=1}^{\infty}$,
\begin{equation*}
\left( \sum_{i=1}^{\infty}\left( \sum_{j=1}^{\infty}|c_{ij}|^{p}\right)
^{q/p}\right) ^{1/q}\leq \left( \sum_{j=1}^{\infty}\left(
\sum_{i=1}^{\infty}|c_{ij}|^{q}\right) ^{p/q}\right) ^{1/p}.
\end{equation*}
\end{proposition}

\section{Recent ``unexpected'' applications to classical problems}

\subsection{The Bohnenblust--Hille inequality with subpolynomial constants}

\label{989}

The Riemann hypothesis certainly motivated and inspired many prestigious
mathematicians from the 20th century to study Dirichlet sums in a more
extensive fashion (for instance, Bourgain, Enflo, or Montgomery \cite{B,E,M}%
). Perhaps, for this reason, in the first decades of the 20th century Harald
Bohr was merged in the study of Dirichlet series (see \cite{bohr1913,
bohr1914, bohr1914b}). One of his main interests was to determine the width
of the maximal strips on which a Dirichlet series can converge absolutely
but non uniformly. More precisely, for a Dirichlet series $%
\sum\limits_{n}a_{n}n^{-s},$ Bohr defined
\begin{equation*}
\sigma _{a}=\inf \left\{ r:\sum_{n}a_{n}n^{-s}\text{ converges for }%
Re(s)>r\right\} ,
\end{equation*}%
\begin{equation*}
\sigma _{u}=\inf \left\{ r:\sum_{n}a_{n}n^{-s}\text{ converges uniformly in }%
Re\left( s\right) >r+\varepsilon \text{ for every }\varepsilon >0\right\} ,
\end{equation*}%
\noindent and
\begin{equation*}
T=\sup \left\{ \sigma _{a}-\sigma _{u}\right\} .
\end{equation*}%
Bohr's question was: What is the value of $T?$

The Bohnenblust--Hille inequality was proved in 1931 by H.F. Bohnenblust and
E. Hille and it is a crucial tool to answer Bohr's problem: They proved that
$T=1/2.$

When dealing with the Bohnenblust--Hille inequality it is elucidative to
begin by proving Littlewood's $4/3$ inequality, a predecessor of the
Bohnenblust--Hille inequality. Littlewood's $4/3$ inequality was proved in
1930 to solve a problem posed by P.J. Daniell. It is worth noticing how
Holder's inequality plays a fundamental role in the argument used in the
proof. We include (for the sake of completeness) a proof of the optimality
of the power $4/3$ using the Kahane-Salem-Zygmund inequality.

\begin{theorem}[Littlewood's $4/3$ inequality]
\label{Littlewood43} There is a constant $\mathrm{L}_{\mathbb{K}}\geq1$ such
that
\begin{equation*}
\left( \sum\limits_{i,j=1}^{N}\left\vert U(e_{i},e_{j})\right\vert ^{\frac {4%
}{3}}\right) ^{\frac{3}{4}}\leq \mathrm{L}_{\mathbb{K}}\left\Vert
U\right\Vert
\end{equation*}
for every bilinear form $U:\ell_{\infty}^{N}\times\ell_{\infty}^{N}%
\rightarrow\mathbb{K}$ and every positive integer $N.$ Moreover, the power $%
4/3$ is optimal.
\end{theorem}

\begin{proof}
Note that%
\begin{equation*}
\sum\limits_{i,j=1}^{N}\left\vert U(e_{i},e_{j})\right\vert ^{\frac{4}{3}%
}\leq \left[ \sum\limits_{i=1}^{N}\left( \sum\limits_{j=1}^{N}\left\vert
U(e_{i},e_{j})\right\vert ^{2}\right) ^{\frac{1}{2}}\right] ^{\frac{2}{3}}%
\left[ \left( \sum\limits_{i=1}^{N}\left( \sum\limits_{j=1}^{N}\left\vert
U(e_{i},e_{j})\right\vert \right) ^{2}\right) ^{\frac{1}{2}}\right] ^{\frac{2%
}{3}}
\end{equation*}%
is a particular case of the procedure from Section 2. Now we just need to
estimate the two factors above. From the Khinchine inequality we have

\begin{align*}
\sum\limits_{i=1}^{N}\left( \sum\limits_{j=1}^{N}\left\vert
U(e_{i},e_{j})\right\vert ^{2}\right) ^{\frac{1}{2}}& \leq \sqrt{2}%
\sum\limits_{i=1}^{N}\int\limits_{0}^{1}\left\vert
\sum\limits_{j=1}^{N}r_{j}(t)U(e_{i},e_{j})\right\vert dt \\
& \leq \sqrt{2}\int\limits_{0}^{1}\sum\limits_{i=1}^{N}\left\vert
U(e_{i},\sum\limits_{j=1}^{N}r_{j}(t)e_{j})\right\vert dt \\
& \leq \sqrt{2}\sup_{t\in \left[ 0,1\right] }\sum\limits_{i=1}^{N}\left\vert
U(e_{i},\sum\limits_{j=1}^{N}r_{j}(t)e_{j})\right\vert \\
& \leq \sqrt{2}\left\Vert U\right\Vert .
\end{align*}%
By symmetry, the same is true swapping $i$ and $j$. From Minkowski's
inequality we have
\begin{equation*}
\left( \sum\limits_{i=1}^{N}\left( \sum\limits_{j=1}^{N}\left\vert
U(e_{i},e_{j})\right\vert \right) ^{2}\right) ^{\frac{1}{2}}\leq
\sum\limits_{j=1}^{N}\left( \sum\limits_{i=1}^{N}\left\vert
U(e_{i},e_{j})\right\vert ^{2}\right) ^{\frac{1}{2}}\leq \sqrt{2}\left\Vert
U\right\Vert
\end{equation*}
and combining all these inequalities we obtain the result with
\begin{equation*}
\mathrm{L}_{\mathbb{K}}=\sqrt{2}.
\end{equation*}%
To prove the optimality of the exponent $4/3$ we can use the
Kahane--Salem--Zygmund inequality. Let $T_{2,N}:\ell _{\infty
}^{N}\rightarrow \mathbb{C}$ be the bilinear form satisfying the multilinear
Kahane--Salem--Zygmund inequality (Theorem \ref{novv1}). Then%
\begin{equation*}
\left( \sum\limits_{i,j=1}^{N}\left\vert T_{2,N}(e_{i},e_{j})\right\vert
^{q}\right) ^{\frac{1}{q}}\leq \sqrt{2}\mathrm{C}N^{\frac{3}{2}}
\end{equation*}%
and thus%
\begin{equation*}
N^{\frac{2}{q}}\leq \sqrt{2}\mathrm{C}N^{\frac{3}{2}}.
\end{equation*}%
Next, letting $N\rightarrow \infty $ we conclude that $q\geq \frac{4}{3}.$
\end{proof}

The natural generalization of Littlewood's $4/3$ inequality is the
Bohnenblust--Hille inequality. This inequality essentially says that for $%
m>2 $ the exponent $\frac{4}{3}$ can be replaced by $\frac{2m}{m+1}$, and
this exponent is optimal. More precisely, it asserts that, for any $m\geq 2$%
, there exists a constant $\mathrm{C}_{\mathbb{K},m}\geq 1$ such that, for
all $m$-linear forms $T:\ell _{\infty }^{N}\times \dots \times \ell _{\infty
}^{N}\rightarrow \mathbb{K}$,
\begin{equation}  \label{eq51}
\left( \sum\limits_{i_{1},\ldots ,i_{m}=1}^{N}\left\vert
T(e_{i_{^{1}}},\ldots ,e_{i_{m}})\right\vert ^{\frac{2m}{m+1}}\right) ^{%
\frac{m+1}{2m}}\leq \mathrm{C}_{\mathbb{K},m} \left\Vert T\right\Vert,
\end{equation}
and all $N$.

This result was overlooked and, sometimes, rediscovered during the last $80$
years. Different approaches led to different values of the constants $%
\mathrm{C}_{m}$. Let us denote the optimal constants satisfying equation %
\eqref{eq51} above by $\mathrm{B}_{\mathbb{K},m}^{\mathrm{mult}}$. As a
matter of fact, controlling the growth of the constants $\mathrm{B}_{\mathbb{%
K},m}^{\mathrm{mult}}$ is crucial for applications, as it is being left
clear along the paper.

Now we show how a suitable use of H\"{o}lder's inequality (Theorem \ref%
{gen.interp}) provides a very simple proof of the Bohnenblust--Hille
inequality, with (so far!) the best known constants.

With the ingredients of Section 4 we can easily obtain an inductive formula
for $\mathrm{B}_{\mathbb{K},m}^{\mathrm{mult}}$. We present a sketch of the
proof (more details can be found in \cite{bps}; we also refer to the excellent survey \cite{sevillap}).

\begin{theorem}[Bohnenblust--Hille inequality]
\label{PROPINDUCMULT} For any positive integer $m$, there exists a constant $%
\mathrm{B}_{\mathbb{K},m}^{\mathrm{mult}}\geq 1$ such that, for all $m$%
-linear forms $L:\ell _{\infty }^{N}\times \dots \times \ell _{\infty
}^{N}\rightarrow \mathbb{K}$ and all $N$,
\begin{equation}
\left( \sum\limits_{i_{1},\ldots ,i_{m}=1}^{N}\left\vert
L(e_{i_{^{1}}},\ldots ,e_{i_{m}})\right\vert ^{\frac{2m}{m+1}}\right) ^{%
\frac{m+1}{2m}}\leq \mathrm{B}_{\mathbb{K},m}^{\mathrm{mult}}\left\Vert
L\right\Vert ,
\end{equation}
with $\mathrm{B}_{\mathbb{K},1}^{\mathrm{mult}}=1$ and $\mathrm{B}_{\mathbb{K%
},m}^{\mathrm{mult}}\leq \mathrm{A}_{\mathbb{K},\frac{2k}{k+1}}^{-1}\mathrm{B%
}_{\mathbb{K},k}^{\mathrm{mult}}.$
\end{theorem}

\begin{proof}
We present a simple proof for the case $k=m-1$, which is the most important,
since it provides better constants (and the proof for other values of $k$ is
similar). The proof for $\mathbb{R}$ is essentially the same as the proof
for $\mathbb{C}$, so we present only the proof for the complex case. Let $%
n\geq 1$ and let $L=\sum_{\mathbf{i}\in \mathbb{N}^{m}}a_{\mathbf{i}%
}z_{i_{1}}^{(1)}\dots z_{i_{m}}^{(m)}$ be an $m$-linear form on $\ell
_{\infty }^{N}\times \dots \times \ell _{\infty }^{N}$.

From the Khinchine inequality we have
\begin{equation*}
\left( \sum_{\mathbf{i}_{S}}\left( \sum_{\mathbf{i}_{\hat{S}}}|a_{\mathbf{i}%
}|^{2}\right) ^{\frac{1}{2}\times \frac{2m-2}{m}}\right) ^{\frac{m}{2m-2}%
}\leq \mathrm{A}_{\mathbb{C},\frac{2m-2}{m}}^{-1}\mathrm{B}_{\mathbb{C}%
,m-1}^{\mathrm{mult}}\Vert L\Vert .
\end{equation*}%
with exponents%
\begin{equation*}
\left( q_{1},\ldots,q_{m}\right) =\left( \frac{2m-2}{m},\ldots,\frac{2m-2}{m}%
,2\right)
\end{equation*}%
From the \textquotedblleft Minkowski inequality\textquotedblright\
(Proposition \ref{8776}) we can obtain analogous estimates if we take the $2$
in the last position and move it backwards making it take every position
from the last to the first; in other words, considering the following
exponents:
\begin{equation*}
\left( \frac{2m-2}{m},\ldots,2,\frac{2m-2}{m}\right) ,\ldots,\left( 2,\frac{%
2m-2}{m},\ldots,\frac{2m-2}{m}\right)
\end{equation*}
and the same constant. Using the H\"{o}lder inequality for multiple
exponents we reach the result.
\end{proof}

Using the values of the constants $\mathrm{A}_{\mathbb{K},p}$ we conclude
that
\begin{equation}
\mathrm{B}_{\mathbb{C},m}^{\mathrm{mult}}\leq \prod\limits_{j=2}^{m}\Gamma
\left( 2-\frac{1}{j}\right)^{\frac{j}{2-2j}}.  \label{uyt}
\end{equation}

For real scalars and $m\geq14$,
\begin{equation}
\mathrm{B}_{\mathbb{R},m}^{\mathrm{mult}}\leq 2^{\frac{446381}{55440}-\frac{m%
}{2}}\prod\limits_{j=14}^{m}\left( \frac{\Gamma \left( \frac{3}{2}-\frac{1}{j%
}\right) }{\sqrt{\pi }}\right) ^{\frac{j}{2-2j}}  \label{uyt5}
\end{equation}
and
\begin{equation*}
\mathrm{B}^{\mathrm{mult}}_{\mathbb{R},m}\leq\prod \limits_{j=2}^{m}2^{\frac{%
1}{2j-2}}= \left(\sqrt{2}\right)^{\sum_{j=1}^{m-1} 1/j}.
\end{equation*}
for $2\leq m\leq13$.

However, a first look at \eqref{uyt} and \eqref{uyt5} gives \textit{a priori
}no clues on their behavior. The following consequences of Theorem \ref%
{PROPINDUCMULT} taken from \cite{bps} are illuminating:

\begin{itemize}
\item There exists $\kappa _{1}>0$ such that, for any $m\geq 1$,
\begin{equation*}
\mathrm{B}_{\mathbb{C},m}^{\mathrm{mult}}\leq \kappa _{1}m^{\frac{1-\gamma }{%
2}}<\kappa _{1}m^{0.211392}.
\end{equation*}

\item There exists $\kappa _{2}>0$ such that, for any $m\geq 1$,
\begin{equation*}
\mathrm{B}_{\mathbb{R},m}^{\mathrm{mult}}\leq \kappa _{2}m^{\frac{2-\log
2-\gamma }{2}}<\kappa _{2}m^{0.36482}.
\end{equation*}
\end{itemize}

It is interesting to note that some old estimates $\mathrm{B}_{\mathbb{K}%
,m}^{\mathrm{mult}}$ can be easily recovered just by choosing different $%
\left( q_{1},\ldots ,q_{m}\right) $ when using H\"{o}lder's inequality (or
using Theorem \ref{PROPINDUCMULT} directly). For instance,

\begin{itemize}
\item Davie (\cite{davie}, 1973).
\begin{equation*}
\mathrm{B}_{\mathbb{K},m}^{\mathrm{mult}} \leq \left( \sqrt{2}\right) ^{m-1}.
\end{equation*}
Using the Khinchine inequality, we have
\begin{equation*}
\left( \sum_{i_{1}=1}^{n}\left( \dots \left( \sum_{i_{m}=1}^{n}|a_{\mathbf{i}%
}|^{q_{m}}\right) ^{\frac{q_{m-1}}{q_{m}}}\dots \right) ^{\frac{q_{1}}{q_{2}}%
}\right) ^{\frac{1}{q_{1}}}\leq \left( \sqrt{2}\right) ^{m-1}\Vert L\Vert
\end{equation*}%
for
\begin{equation*}
\left( q_{1},\ldots ,q_{m}\right) =\left( 1,2,\ldots ,2\right)
\end{equation*}%
Using the \textquotedblleft Minkowski inequality\textquotedblright\
(Proposition \ref{8776}) we get the same estimate for%
\begin{equation*}
\left( q_{1},\ldots ,q_{m}\right) =\left( 2,1,\ldots ,2\right) ,\ldots
,\left( q_{1},\ldots ,q_{m}\right) =\left( 2,\ldots ,2,1\right)
\end{equation*}%
with the same constant. Now, using Theorem \ref{gen.interp} we conclude the
proof with
\begin{equation*}
\mathrm{B}_{\mathbb{K},m}^{\mathrm{mult}}\leq \left( \sqrt{2}\right) ^{m-1}.
\end{equation*}

\item Pellegrino and Seoane-Sep\'ulveda (\cite{pseo}, 2012).
\begin{eqnarray*}
\mathrm{B}_{\mathbb{K},m}^{\mathrm{mult}} &\leq &\mathrm{A}_{\mathbb{K},%
\frac{2m}{m+2}}^{-m/2}\mathrm{B}_{\mathbb{K},m/2}^{\mathrm{mult}}\text{ for }%
m~\text{even, and} \\
\mathrm{B}_{\mathbb{K},m}^{\mathrm{mult}} &\leq &\left( \mathrm{A}_{\mathbb{K%
},\frac{2m-2}{m+1}}^{\frac{-m-1}{2}}\mathrm{B}_{\mathbb{K},\frac{m-1}{2}}^{%
\mathrm{mult}}\right) ^{\frac{m-1}{2m}}\left( \mathrm{A}_{\mathbb{K},\frac{%
2m+2}{m+3}}^{\frac{1-m}{2}}\mathrm{B}_{\mathbb{K},\frac{m+1}{2}}^{\mathrm{%
mult}}\right) ^{\frac{m+1}{2m}}\text{, for }m\text{ odd.}
\end{eqnarray*}
\end{itemize}

When $m$ is even and $k=m/2$, we use Khinchine inequality to obtain
estimates for the inequalities with the exponent%
\begin{equation*}
\left( q_{1},\ldots ,q_{m}\right) =\left( \frac{2m}{m+2},\ldots ,\frac{2m}{%
m+2},2,\ldots ,2\right) \text{ }
\end{equation*}
and using the Minkowski inequality the same estimate is obtained for%
\begin{equation*}
\left( q_{1},\ldots ,q_{m}\right) =\left( 2,\ldots ,2,\frac{2m}{m+2},\ldots ,%
\frac{2m}{m+2}\right) .\text{ }
\end{equation*}%
Using Proposition \ref{PROPINDUCMULT} we obtain
\begin{equation*}
\mathrm{B}_{\mathbb{K},m}^{\mathrm{mult}}\leq \mathrm{A}_{\mathbb{K},\frac{2m%
}{m+2}}^{-m/2}\mathrm{B}_{\mathbb{K},m/2}^{\mathrm{mult}}.
\end{equation*}%
The case $m$ odd is somewhat similar, although it needs a little \emph{trick}%
. It is worth mentioning that these estimates from \cite{pseo} can be
somewhat derived from abstract results appearing in \cite{defa}.

\subsection{Quantum Information Theory}

Here we shall briefly describe a result by Montanaro \cite[Theorem 5]%
{Montanaro} which provided an application for the optimal Bohnenblust-Hille
constants within the field of Quantum Physics. This presentation is based on
Schwarting's Ph.D. dissertation \cite[Section 2.2.5]{ursula}. For a more
detailed information we refer the interested reader to the Ph.D.
dissertation of Bri\"et \cite[Chapter 1]{briet}, which provides a very clear
introduction to the whole topic of nonlocal games.

A classical nonlocal game is a pair $\mathcal{G} = \left(A,\pi\right)$
consisting on a function (called \emph{predicate}) $A : \mathcal{A} \times
\mathcal{B} \times \mathcal{S} \times \mathcal{S} \to \{\pm1\}$ and a
probability distribution $\pi : \mathcal{S} \times \mathcal{T} \to [0,1]$.
The game involves three parties: a person called the \emph{referee} and two
\emph{players} (usually called Alice and Bob). When the game starts, the
referee picks a question $(s,t) \in \mathcal{S} \times \mathcal{T}$
according to the probability distribution $\pi$ and, then, sends it to Alice
and Bob, who must reply independently (they are not allowed to communicate
between each other once the game has begun) by providing an answer $a\in%
\mathcal{A}$ and $b\in\mathcal{B}$ each one. The players win the game if $%
A(a,b,s,t)=1$, and lose otherwise. The players' goal is to maximize their
chance of winning. A XOR game is a nonlocal game in which the answer sets $%
\mathcal{A},\mathcal{B}$ are $\{\pm1\}$ and the predicate $A$ depends only
on the exclusive-OR (XOR) of the answers given by the players and the value
of a Boolean function $\mathcal{S} \times \mathcal{T} \to \{\pm1\}$, which
from the predicate may be seen as a matrix with entries on $\{\pm1\}$. A
game with $m$-players is described similarly in the following fashion.

An $m$-player XOR (exclusive OR) game is a pair $\mathcal{G}=\left(
\pi,A\right) $ consisting of a matrix $A=\left( a_{\mathbf{i}}\right) _{%
\mathbf{i}\in \mathcal{M}(m,n)}$, for which each entry $a_{\mathbf{i}}\in
\{\pm 1\}$, and a probability distribution $\pi :\mathcal{M}(m,n)\rightarrow
\lbrack 0,1]$. The game consists on \emph{the referee} picking an $m$-tuple $%
\mathbf{i}=(i_{1},\dots ,i_{m})\in \mathcal{M}(m,n)$ according to the
probability distribution $\pi $ and sending each question $i_{k}$ to \emph{%
the player} $k$, which, by means of a classical strategy, must reply upon
this question with a (deterministic) answer map $y_{k}:\{1,\dots
,n\}\rightarrow \{\pm 1\}$. The players win if and only if the product of
their answers equals the corresponding entry in the matrix $A$, that is if
\begin{equation*}
y_{1}(i_{1})\cdots y_{m}(i_{m})=a_{\mathbf{i}}.
\end{equation*}%
Concerning the complexity of a XOR game, one defines the bias $\beta (G)$ to
be the greatest difference between the chance of winning and the chance of
loosing the game for the optimal classical strategy. Therefore, the
classical bias of an $m$-player XOR game is given by
\begin{equation*}
\beta (G)=\max_{y_{1},\dots ,y_{m}\in \{\pm 1\}^{n}} \left\vert \sum_{%
\mathbf{i}\in \mathcal{M}(m,n)}\pi (\mathbf{i})a_{\mathbf{i}%
}y_{1}(i_{1})\cdots y_{m}(i_{m})\right\vert .
\end{equation*}
If we define the $m$-linear map $T:\ell _{\infty }^{n}\times \dots \times
\ell _{\infty }^{n}\rightarrow \mathbb{R}$ by $T(e_{i_{1}},\dots
,e_{i_{m}}):=a_{\mathbf{i}}\pi (\mathbf{i})$, then the bias will be
\begin{equation*}
\beta (G)=\Vert T\Vert .
\end{equation*}

A natural problem is to find the game for which the classical bias is
minimized. It is known that there exists an $m$-player XOR game $\mathcal{G}$
for which
\begin{equation*}
\beta(\mathcal{G}) \leq n^{-\frac{m-1}{2}}
\end{equation*}
(see \cite{FordGal}). Using the Bohnenblust-Hille inequality it is
straightforward to obtain lower bounds for the classical bias of an $m$%
-player XOR games (see \cite[Theorem 5]{Montanaro}).

\begin{theorem}
{\cite[Theorem 5]{Montanaro}} For every $m$-player XOR game $\mathcal{G} =
(\pi,A)$,
\begin{equation*}
\beta(\mathcal{G}) \geq \frac{1}{\kappa m^{^{0.36482}}}n^{\frac{1-m}{2}},
\end{equation*}
where $\kappa>0$ is an universal constant.
\end{theorem}

\begin{proof}
Define the $m$-linear form $T:\ell _{\infty }^{n}\times \dots \times \ell
_{\infty }^{n}\rightarrow \mathbb{R}$ by $T(e_{i_{1}},\dots ,e_{i_{m}}):=a_{%
\mathbf{i}}\pi (\mathbf{i})$. Then,
\begin{equation*}
\sum_{\mathbf{i}\in \mathcal{M}(m,n)}\left\vert T(e_{i_{1}},\dots
,e_{i_{m}})\right\vert =\sum_{\mathbf{i}\in \mathcal{M}(m,n)}\pi (\mathbf{i}%
)=1.
\end{equation*}%
Applying H\"{o}lder's inequality and the Bohnenblust-Hille, we conclude that
\begin{align*}
\sum_{\mathbf{i}\in \mathcal{M}(m,n)}\left\vert T(e_{i_{1}},\dots,
e_{i_{m}})\right\vert \leq & \left( \sum_{\mathbf{i}\in \mathcal{M}%
(m,n)}\left\vert T(e_{i_{1}},\dots ,e_{i_{m}})\right\vert ^{\frac{2m}{m+1}%
}\right) ^{\frac{m+1}{2m}}\left( \sum_{\mathbf{i}\in \mathcal{M}%
(m,n)}1\right) ^{\frac{m-1}{2m}} \\
\leq & \mathrm{B}_{\mathbb{R},m}^{\mathrm{mult}}n^{\frac{m-1}{2}}\Vert
T\Vert =\mathrm{B}_{\mathbb{R},m}^{\mathrm{mult}}n^{\frac{m-1}{2}}\beta (%
\mathcal{G}).
\end{align*}

Using the best known estimates for the multilinear Bohnenblust--Hille
inequality we conclude that
\begin{equation*}
\beta (\mathcal{G})\geq \frac{1}{\kappa m^{\frac{2-\log 2-\gamma }{2}}n^{%
\frac{m-1}{2}}}>\frac{1}{\kappa m^{^{0.36482}}}n^{\frac{1-m}{2}}.
\end{equation*}
\end{proof}

This result, according to Montanaro (see \cite[p.4]{Montanaro}), implies a
very particular case of a conjecture of Aaronson and Ambainis (see \cite{aaa}%
). Also, recent advances on the real polynomial Bohnenblust-Hille inequality
(see, e.g., \cite{cjmps2014,pams2014}), combined with the CHSH inequality
(due to Clauser, Horne, Shimony, and Holt in the late 1960's), can be
employed in the proof of Bell's theorem, which states that certain
consequences of entanglement in quantum mechanics cannot be reproduced by
local hidden variable theories. We refer the interested reader to the
seminal paper, \cite{CHSH}, in which more informtaion regarding this CHSH
inequality can be found.

\subsection{Power series and the Bohr radius problem \label{811}}

The following question was addressed by H. Bohr in 1914:

\begin{quote}
\emph{How large can the sum of the mudulii of the terms of a convergent
power series be?}
\end{quote}

\noindent The answer was given by the following theorem, which was
independently obtain by Bohr, Riesz, Schur, and Wiener:

\begin{theorem}
\label{th54} Suppose that a power series $\sum_{k=0}^{\infty }c_{k}z^{k}$
converges for $z $ in the unit disk, and $\left\vert \sum_{k=0}^{\infty
}c_{k}z^{k}\right\vert <1$ when $\left\vert z\right\vert <1.$ Then $%
\sum_{k=0}^{\infty }\left\vert c_{k}z^{k}\right\vert <1$ when $\left\vert
z\right\vert <1/3.$ Moreover, the radius $1/3$ is the best possible.
\end{theorem}

Following Boas and Khavinson \cite{BK97}, the Bohr radius $\mathrm{K}_{n}$
of the $n$-dimensional polydisk is the largest positive number $r$ such that
all polynomials $\sum_{\alpha }a_{\alpha }z^{\alpha }$ on $\mathbb{C}^{n}$
satisfy
\begin{equation*}
\sup_{z\in r\mathbb{D}^{n}}\sum_{\alpha }|a_{\alpha }z^{\alpha }|\leq
\sup_{z\in \mathbb{D}^{n}}\left\vert \sum_{\alpha }a_{\alpha }z^{\alpha
}\right\vert .
\end{equation*}%
The Bohr radius $\mathrm{K}_{1}$ was estimated by H. Bohr, M. Riesz, I.
Schur and F. Wiener, and it was shown that $\mathrm{K}_{1}=1/3$ (Theorem \ref%
{th54}). For $n\geq 2$, exact values of $\mathrm{K}_{n}$ are unknown. In
\cite{BK97}, it was proved that%
\begin{equation}
\frac{1}{3}\sqrt{\frac{1}{n}}\leq \mathrm{K}_{n}\leq 2\sqrt{\frac{\log n}{n}}%
.  \label{9788}
\end{equation}

The paper by Boas and Khavinson, \cite{BK97}, motivated many other works,
connecting the asymptotic behavior of $\mathrm{K}_{n}$ to various problems
in Functional Analysis (geometry of Banach spaces, unconditional basis
constant of spaces of polynomials, etc.); we refer to \cite{DP06} for a
panorama of the subject. Hence there was a big motivation in recent years in
determining the behavior of $\mathrm{K}_{n}$ for large values of $n$.

In \cite{DeFr06}, the lefthand side inequality of (\ref{9788}) was improved
to
\begin{equation*}
\mathrm{K}_{n}\geq c\sqrt{\log n/(n\log \log n)}.
\end{equation*}
In \cite{annals}, using the hypercontractivity of the polynomial
Bohnenblust--Hille inequality, the authors showed that%
\begin{equation}
\mathrm{K}_{n}=b_{n}\sqrt{\frac{\log n}{n}}\text{ with }\frac{1}{\sqrt{2}}%
+o(1)\leq b_{n}\leq 2.  \label{1212}
\end{equation}%
In this section we sketch how the H\"{o}lder inequality for mixed sums
played a fundamental role in the final answer to the solution, given in \cite%
{bps}, to the Bohr radius problem:
\begin{equation*}
\lim_{n\rightarrow \infty }\frac{\mathrm{K}_{n}}{\sqrt{\frac{\log n}{n}}}=1.
\end{equation*}

The solution has several ingredients, including the polynomial
Bohnenblust--Hille inequality. Using (\ref{polar}), Bohnenblust and Hille
were also able to have a polynomial version of this inequality: for any $%
m\geq 1$, there exists a constant $\mathrm{D}_{m}\geq 1$ such that, for any
complex $m$-homogeneous polynomial $P(\mathbf{z})=\sum_{|\alpha
|=m}a_{\alpha }\mathbf{z}^{\alpha }$ on $c_{0}$,
\begin{equation*}
\left( \sum_{|\alpha |=m}|a_{\alpha }|^{\frac{2m}{m+1}}\right) ^{\frac{m+1}{%
2m}}\leq \mathrm{D}_{m}\Vert P\Vert,
\end{equation*}
with
\begin{equation*}
\mathrm{D}_{m}=\left( \sqrt{2}\right) ^{m-1}\frac{m^{\frac{m}{2}}\left(
m+1\right) ^{\frac{m+1}{2}}}{2^{m}\left( m!\right) ^{\frac{m+1}{2m}}}.
\end{equation*}

In fact, it is not difficult to use polarization and obtain the polynomial
Bohnenblust-Hille inequality by using the multilinear Bohnenblust-Hille
inequality, but with \emph{bad} constants (the following approach can be
essentially found in \cite[Lemma 5]{defant2}). In fact, if $L$ is the polar
of $P$, from (\ref{800}) we have
\begin{align*}
{\sum\limits_{\left\vert \alpha \right\vert =m}}\left\vert a_{\alpha
}\right\vert ^{\frac{2m}{m+1}}& ={\sum\limits_{\left\vert \alpha \right\vert
=m}}\left( \binom{m}{\alpha }\left\vert L(e_{1}^{\alpha _{1}},\ldots
,e_{n}^{\alpha _{n}})\right\vert \right) ^{\frac{2m}{m+1}} \\
& ={\sum\limits_{\left\vert \alpha \right\vert =m}}\binom{m}{\alpha }^{\frac{%
2m}{m+1}}\left\vert L(e_{1}^{\alpha _{1}},\ldots ,e_{n}^{\alpha
_{n}})\right\vert ^{\frac{2m}{m+1}}.
\end{align*}%
However, for every choice of $\alpha $, the term
\begin{equation*}
\left\vert L(e_{1}^{\alpha _{1}},\ldots ,e_{n}^{\alpha _{n}})\right\vert ^{%
\frac{2m}{m+1}}
\end{equation*}%
is repeated $\binom{m}{\alpha }$ times in the sum
\begin{equation*}
{\sum\limits_{i_{1},\ldots ,i_{m}=1}^{n}}\left\vert L(e_{i_{1}},\ldots
,e_{i_{m}})\right\vert ^{\frac{2m}{m+1}}.
\end{equation*}%
Thus
\begin{equation*}
{\sum\limits_{\left\vert \alpha \right\vert =m}}\binom{m}{\alpha }^{\frac{2m%
}{m+1}}\left\vert L(e_{1}^{\alpha _{1}},\ldots ,e_{n}^{\alpha
_{n}})\right\vert ^{\frac{2m}{m+1}}={\sum\limits_{i_{1},\ldots ,i_{m}=1}^{n}}%
\binom{m}{\alpha }^{\frac{2m}{m+1}}\frac{1}{\binom{m}{\alpha }}\left\vert
L(e_{i_{1}},\ldots ,e_{i_{m}})\right\vert ^{\frac{2m}{m+1}}
\end{equation*}%
and, since
\begin{equation*}
\binom{m}{\alpha }\leq m!
\end{equation*}%
we have
\begin{equation*}
{\sum\limits_{\left\vert \alpha \right\vert =m}}\binom{m}{\alpha }^{\frac{2m%
}{m+1}}\left\vert L(e_{1}^{\alpha _{1}},\ldots ,e_{n}^{\alpha
_{n}})\right\vert ^{\frac{2m}{m+1}}\leq \left( m!\right) ^{\frac{m-1}{m+1}}{%
\sum\limits_{i_{1},\ldots ,i_{m}=1}^{n}}\left\vert L(e_{i_{1}},\ldots
,e_{i_{m}})\right\vert ^{\frac{2m}{m+1}}.
\end{equation*}%
We thus have
\begin{align*}
\left( {\sum\limits_{\left\vert \alpha \right\vert =m}}\left\vert a_{\alpha
}\right\vert ^{\frac{2m}{m+1}}\right) ^{\frac{m+1}{2m}}& \leq \left( \left(
m!\right) ^{\frac{m-1}{m+1}}{\sum\limits_{i_{1},\ldots ,i_{m}=1}^{n}}%
\left\vert L(e_{i_{1}},\ldots ,e_{i_{m}})\right\vert ^{\frac{2m}{m+1}%
}\right) ^{\frac{m+1}{2m}} \\
& =\left( m!\right) ^{\frac{m-1}{2m}}\left( {\sum\limits_{i_{1},\ldots
,i_{m}=1}^{n}}\left\vert L(e_{i_{1}},\ldots ,e_{i_{m}})\right\vert ^{\frac{2m%
}{m+1}}\right) ^{\frac{m+1}{2m}} \\
& \leq \left( m!\right) ^{\frac{m-1}{2m}}\mathrm{B}_{\mathbb{R},m}^{\mathrm{%
mult}}\left\Vert L\right\Vert .
\end{align*}%
On the other hand, since $\left\Vert L\right\Vert \leq \frac{m^{m}}{m!}%
\left\Vert P\right\Vert $ we obtain
\begin{align*}
\left( {\sum\limits_{\left\vert \alpha \right\vert =m}}\left\vert a_{\alpha
}\right\vert ^{\frac{2m}{m+1}}\right) ^{\frac{m+1}{2m}}& \leq \mathrm{B}_{%
\mathbb{R},m}^{\mathrm{mult}}\left( m!\right) ^{\frac{m-1}{2m}}\frac{m^{m}}{%
m!}\left\Vert P\right\Vert \\
& =\mathrm{B}_{\mathbb{R},m}^{\mathrm{mult}}\frac{m^{m}}{\left( m!\right) ^{%
\frac{m+1}{2m}}}\left\Vert P\right\Vert .
\end{align*}

Let us denote the best constant $\mathrm{D}_{m}$ in this inequality by $%
\mathrm{B}_{\mathbb{C},m}^{\mathrm{pol}}$. In \cite{annals} it was proved
that in fact these estimates could be essentially improved to $\left( \sqrt{2%
}\right) ^{m-1}.$ However using the variant of H\"{o}lder's inequality for
mixed $\ell _{p}$ spaces, together with some results from Complex Analysis
(see \cite{bps} for details) and with the subpolynomial estimates of the
multilinear Bohnenblust--Hille inequality (Section 5), one of the main
results of \cite{bps} shows that we can go much further:

\begin{theorem}
\label{i8} For any $\varepsilon >0$, there exists $\kappa >0$ such that, for
any $m\geq 1$,
\begin{equation*}
\mathrm{B}_{\mathbb{C},m}^{\mathrm{pol}}\leq \kappa (1+\varepsilon )^{m}.
\end{equation*}
\end{theorem}

As we mentioned above, in \cite{annals}, using the hypercontractivity of the
polynomial Bohnenblust--Hille inequality, the authors showed that
\begin{equation}
\mathrm{K}_{n}=b_{n}\sqrt{\frac{\log n}{n}}\text{ with }\frac{1}{\sqrt{2}}%
+o(1)\leq b_{n}\leq 2.  \label{uu7}
\end{equation}
However, although \eqref{uu7} is quite precise, there was still uncertainity
in the behavior of the number $b_{n}$. By combining classical tools of
Complex Analysis (Harris' inequality \cite{Harris}), Bayart's inequality
\cite{bbbaaa}, Wiener's inequality \cite[Lemma 6.1]{bps}, and the
Kahane--Salem--Zygmund inequality (Theorem \ref{novv2}) together with
Theorem \ref{i8} the authors, in \cite{bps}, were finally able to provide
the final solution to the Bohr radius problem:

\begin{theorem}
The asymptotic growth of the $n-$dimensional Bohr radius is $\sqrt{\frac{%
\log n}{n}}$. In other words, $\displaystyle \lim_{n\rightarrow \infty }%
\frac{\mathrm{K}_{n}}{\sqrt{\frac{\log n}{n}}}=1.$
\end{theorem}

The crucial step to complete the proof was the improvement of the estimates
of the polynomial Bohnenblust--Hille inequality that was only achieved by
means of the H\"{o}lder inequality for mixed sums.

\subsection{Hardy--Littlewood's inequality constants}

Although H\"{o}lder's inequality for mixed $\ell _{p}$ spaces dates back to
the 1960's, its full importance in the subjects mentioned throughout this
paper was just very recently realized. New consequences are still appearing
(see, for instance \cite{joedson,aps,carando}). The last applications of the
H\"{o}lder inequality for mixed $\ell _{p}$ spaces presented here concern
the Hardy--Littlewood inequality and the theory of multiple summing
multilinear operators. As in the case of the Bohnenblust--Hille inequality
(Section 5) the H\"{o}lder inequality for multiple exponents allows a
significant improvement in the constants of the Hardy--Littlewood inequality.

Let $\mathbb{K}$ be $\mathbb{R}$ or $\mathbb{C}$. Given an integer $m\geq 2$%
, the Hardy--Littlewood inequality (see \cite{alb, hardy,pra}) asserts that
for $2m\leq p\leq \infty $ there exists a constant $\mathrm{C}_{m,p}^{%
\mathbb{K}}\geq 1$ such that, for all continuous $m$--linear forms $T:\ell
_{p}^{n}\times \cdots \times \ell _{p}^{n}\rightarrow \mathbb{K}$ and all
positive integers $n$,
\begin{equation}
\left( \sum_{j_{1},\ldots ,j_{m}=1}^{n}\left\vert T(e_{j_{1}},\ldots
,e_{j_{m}})\right\vert ^{\frac{2mp}{mp+p-2m}}\right) ^{\frac{mp+p-2m}{2mp}%
}\leq \mathrm{C}_{m,p}^{\mathbb{K}}\left\Vert T\right\Vert .  \label{i99}
\end{equation}%
Using the generalized Kahane-Salem-Zygmund inequality (see \cite{alb}) one
can easily verify that the exponents $\frac{2mp}{mp+p-2m}$ are optimal. When
$p=\infty,$ using that $\frac{2mp}{mp+p-2m}=\frac{2m}{m+1}$, we recover the
classical Bohnenblust--Hille inequality (see Theorem \ref{PROPINDUCMULT} and
\cite{bh}).

From \cite{bps} we know that $\mathrm{B}_{\mathbb{K},m}^{\mathrm{mult}}$ has
a subpolynomial growth. On the other hand, the best known upper bounds for
the constants in (\ref{i99}) were, until just recently, $\left( \sqrt{2}%
\right) ^{m-1}$ (see \cite{alb, n, dimant}). Although, a suitable use of
Theorem \ref{gen.interp} shows that $\left( \sqrt{2}\right) ^{m-1}$ can be
improved (see \cite{aps}) to
\begin{equation*}
\mathrm{C}_{m,p}^{\mathbb{R}}\leq \left( \sqrt{2}\right) ^{\frac{2m\left(
m-1\right) }{p}}\left( \mathrm{B}_{\mathbb{R},m}^{\mathrm{mult}}\right) ^{%
\frac{p-2m}{p}}
\end{equation*}%
for real scalars and to
\begin{equation*}
\mathrm{C}_{m,p}^{\mathbb{C}}\leq \left( \frac{2}{\sqrt{\pi }}\right) ^{%
\frac{2m(m-1)}{p}}\left( \mathrm{B}_{\mathbb{C},m}^{\mathrm{mult}}\right) ^{%
\frac{p-2m}{p}}
\end{equation*}
for complex scalars. These estimates are substantially better than $\left(%
\sqrt{2}\right) ^{m-1}$ because $\mathrm{B}_{\mathbb{K},m}^{\mathrm{mult}}$
has a subpolynomial growth. In particular, if $p>m^{2}$ we conclude that $%
\mathrm{C}_{m,p}^{\mathbb{K}}$ has a subpolynomial growth.

\subsection{Separately summing operators}

H\"{o}lder's inequality is also used to generalize recent results on the
theory of multiple summing multilinear operators. In \cite{defa}, and for $m$%
-linear operators on $q$-cotype Banach spaces, the authors introduced the
notion \emph{separately $(r,1)$-summing}, with $1\leq r\leq q<\infty$, which
means that, for any $(m-1)$-coordinates fixed, the resulting linear operator
is $(r,1)$-summing. Using separately summing maps, the authors concluded
that the initial operator is multiple $\left( \frac{qrm}{q+\left(m-1\right)r}%
, 1\right)$-summing. In \cite{joedson} it is presented the concept of $n$%
-separability summing, which stands for the $m$-linear operators that are
multiple summing in $n$-coordinates, when there are $m-n$ other coordinates
fixed. Using suitable interpolation, the authors provide $N$-separability
from $n$-separability summing, with $n<N\leq m$. This result generalizes the
previous one and provide more efficient exponents in some special cases.
Moreover, it is also useful to provides estimates for the constants of some
variation of Bohnenblust-Hille inequalities introduced in \cite[Appendix A]%
{ddss} and \cite{111}.


\end{document}